\documentclass[10pt]{article}

\usepackage{hyperref}
\usepackage{xcolor}
\usepackage{amssymb}
\usepackage{amsfonts}
\usepackage{amsthm}
\usepackage{amsmath}
\usepackage{float}
\usepackage{bigints}
\usepackage{fullpage}
\usepackage{mathtools}
\usepackage[utf8]{inputenc} 
\usepackage{cancel}
\usepackage{verbatim}

\newtheorem{theorem}{Theorem}

\newtheorem{lemma}[theorem]{Lemma}

\newtheorem{definition}{Definition}

\newtheorem{corollary}[theorem]{Corollary}
\newtheorem{remark}[theorem]{Remark}

\newcommand{\inter}{\rm int}

\newcommand{\x}{\mathbf{x}}
\newcommand{\p}{\mathbf{p}}
\newcommand{\q}{\mathbf{q}}

\newcommand{\oo}{\mathbf{o}}

\newcommand{\B}{\mathbf{B}}

\newcommand{\A}{\mathbf{A}}
\newcommand{\K}{\mathbf{K}}

\newcommand{\Ee}{\mathbb{E}}
\newcommand{\Hh}{\mathbb{H}}
\newcommand{\Ss}{\mathbb{S}}
\newcommand{\Mm}{\mathbb{M}}

\newcommand{\Ed}{\Ee^d}
\newcommand{\Sd}{\Ss^d}
\newcommand{\Hd}{\Hh^d}

\newcommand{\iprod}[2]{\left<#1,#2\right>}

\newcommand{\noshow}[1]{}
\newcommand{\ivol}[2][k]{{\rm V}_{#1}\left(#2\right)}
\newcommand{\kind}{k\in[d]}

\title{On the intrinsic volumes of intersections of congruent balls
\footnote{Keywords and phrases: Kneser-Poulsen conjecture, Gromov-Klee-Wagon problem, 
(intrinsic) volume of intersections of congruent balls, $r$-ball body, uniform contraction, Euclidean space. \newline \hspace*{.35cm} 2010 Mathematics Subject Classification: 52A20, 52A22.}}

\author{K\'{a}roly Bezdek\thanks{Partially supported by a Natural Sciences and 
Engineering Research Council of Canada Discovery Grant.}
}

\date{}

\begin{document}

\maketitle

\begin{abstract}
Let $\Ee^d$ denote the $d$-dimensional Euclidean space. The $r$-ball body generated by a given set in $\Ee^d$ is the intersection of balls of radius $r$ centered at the points of the given set. In this paper we prove the following Blaschke-Santal\'o-type inequalities for $r$-ball bodies: for all $1\leq k\leq d$ and for any set of given volume in $\Ee^d$ the $k$-th intrinsic volume of the $r$-ball body generated by the set becomes maximal if the set is a ball.
As an application we investigate the Gromov-Klee-Wagon problem for congruent balls in $\Ee^d$, which is a question on proving or disproving that if the centers of a family of $N$ congruent balls in $\Ee^d$ are contracted, then the volume of the intersection does not decrease. In particular, we investigate this problem for uniform contractions, which are contractions where all the pairwise distances in the first set of centers are larger than all the pairwise distances in the second set of centers, that is, when the pairwise distances of the two sets are separated by some positive real number. The author and M. Nasz\'odi [Discrete Comput. Geom. 60/4 (2018), 967-980] proved that the intrinsic volumes of the intersection of $N$ congruent balls in $\Ee^d$, $d>1$ increase under any uniform contraction of the center points when $ N\geq \left(1+\sqrt{2}\right)^d$. We give a short proof of this result using the Blaschke-Santal\'o-type inequalities of $r$-ball bodies and improve it for $d\geq 42$.
\end{abstract}

\section{Introduction}\label{sec:intro}

We denote the Euclidean norm of a vector $\p$ in the $d$-dimensional Euclidean space $\Ee^d$ by $|\p|:=\sqrt{\iprod{\p}{\p}}$, where $\iprod{\cdot}{\cdot}$ is the 
standard inner product. For a positive integer $N$, we use $[N]:=\{1,2,\ldots,N\}$. Let $A\subset\Ed$ be a compact convex set, and 
$\kind$. We denote the $k$-th intrinsic volume of $A$ by $\ivol{A}$. It is well known that $\ivol[d]{A}$ is the $d$-dimensional 
volume of $A$, $2\ivol[d-1]{A}$ is the surface area of $A$, and $\frac{2\omega_{d-1}}{d\omega_d}\ivol[1]{A}$ is equal to the mean width of $A$, where $\omega_d$ stands for the volume of a $d$-dimensional unit ball, that is, $\omega_d=\frac{\pi^{\frac{d}{2}}}{\Gamma(1+\frac{d}{2})}$. (For a focused overview on intrinsic volumes see \cite{GHSch}). In this paper, for simplicity $\ivol{\emptyset}=0$ for all $\kind$. The closed Euclidean ball of radius $r$ centered at $\p\in\Ed$ is denoted by $\B^d[\p,r]:=\{\q\in\Ed\ |\  |\p-\q|\leq r\}$. Now, we are ready to introduce the central notion of this paper.

\begin{definition}\label{r-dual-body}
For a set $\emptyset\neq X\subseteq\Ee^d$, $d>1$ and $r>0$ let the {\rm $r$-ball body} $X^r$ generated by $X$ be defined by $X^r:=\bigcap_{\x\in X}\B^d[\x, r]$. 
\end{definition}

We note that either $X^r=\emptyset$, or $X^r$ is a point, or ${\inter} (X^r)\neq\emptyset$. Perhaps not surprisingly, $r$-ball bodies of $\Ee^d$ have already been investigated in a number of papers however, under various names such as ``\"uberkonvexe Menge'' (\cite{Ma}), ``$r$-convex domain'' (\cite{Fe}), ``spindle convex set'' (\cite{BLNP}, \cite{KMP}), ``ball convex set'' (\cite{LNT}), ``hyperconvex set'' ({\cite{FKV}), and ``$r$-dual set'' (\cite{Be18}). $r$-ball bodies satisfy some basic identities such as $\left((X^r)^r\right))^r=X^r$ and $(X \cup Y)^r=X^r\cap Y^r$, which hold for any $X\subseteq\Ee^d$ and $Y\subseteq\Ee^d$. Clearly, also monotonicity holds namely, $X\subseteq Y\subseteq\Mm^d$ implies $Y^r\subseteq X^r$. 
In this paper we investigate volumetric relations between $X^r$ and $X$ in $\Ee^d$. First, recall the recent theorem of Gao, Hug, and Schneider \cite{GHSch} stating that for any convex body of given volume in $\Ss^d$ the volume of the spherical polar body becomes maximal if the convex body is a ball. The author has proved the following Euclidean analogue of their theorem in \cite{Be18}. Let $\A\subset\Ee^d$, $d>1$ be a compact set of volume $V_{d}(\A)>0$ and $r>0$. If $\B\subseteq\Ee^d$ is a ball with $V_{d}(\A)=V_{d}(\B)$, then 
\begin{equation}\label{Bezdek-inequality}
V_{d}(\A^r)\leq V_{d}(\B^r).
\end{equation} 
As the theorem of Gao, Hug, and Schneider \cite{GHSch} is often called a spherical counterpart of the Blaschke--Santal\'o inequality, one may refer to the Euclidean analogue (\ref{Bezdek-inequality}) of their theorem as a Blaschke--Santal\'o-type inequality for $r$-ball bodies in $\Ee^d$. As a first result, we extend (\ref{Bezdek-inequality}) to intrinsic volumes by proving the following Blaschke--Santal\'o-type inequalities of $r$-ball bodies in $\Ee^d$. 

\begin{theorem}\label{Bezdek-inequality-extended}
Let $\A\subset\Ee^d$, $d>1$ be a compact set of volume $V_{d}(\A)>0$ and $r>0$. If $\B\subset\Ee^d$ is a ball with $V_{d}(\A)=V_{d}(\B)$, then 
\begin{equation}\label{Bezdek-inequality-generalized}
V_{k}(\A^r)\leq V_{k}(\B^r)
\end{equation}
holds for all $\kind$.
\end{theorem}

The author thanks to one of the referees for noting that Theorem~\ref{Bezdek-inequality-extended} follows from a stochastic version proved by G. Paouris and P. Pivovarov (see Theorem 3.1 in \cite{PaPi}).

As a second result, we discuss the following application of Theorem~\ref{Bezdek-inequality-extended}. We say that the (labeled) point set $Q:=\{\q_1,\dots ,\q_N\}\subset\Ee^d$ is a {\it contraction} of the (labeled) point set $P:=\{\p_1,\dots ,\p_N\}\subset\Ee^d$ in $\Ee^d$, $d>1$ if  $|\q_i-\q_j|\leq |\p_i-\p_j|$ holds for all $1\leq i<j\leq N$. In 1955, M. Kneser \cite{Kn55} and E. T. Poulsen \cite{Po54} independently stated the conjecture that if $Q=\{\q_1,\ldots,\q_N\}$ is a contraction of $P=\{\p_1,\ldots, \p_N\}$ in $\Ed$, $d>1$, then 
\begin{equation}\label{K-P}
\ivol[d]{\bigcup_{i=1}^{N}\mathbf{B}^d[\p_i, r]}\ge  \ivol[d]{\bigcup_{i=1}^{N}\mathbf{B}^d[\q_i, r]}
\end{equation}
holds for all $N>1$ and $r>0$. It is customary to assign also the following related conjecture to M. Kneser and E. T. Poulsen. If $Q=\{\q_1,\ldots,\q_N\}$ is a contraction of $P=\{\p_1,\ldots, \p_N\}$ in $\Ed$, $d>1$, then 
\begin{equation}\label{G-K-W}
V_{d}(P^{r})=\ivol[d]{\bigcap_{i=1}^{N}\mathbf{B}^d[\p_i, r]}\le  \ivol[d]{\bigcap_{i=1}^{N}\mathbf{B}^d[\q_i, r]}=V_{d}(Q^{r})
\end{equation}
holds for all $N>1$ and $r>0$. However, a closer look of the relevant literature reveals that the question on proving (\ref{G-K-W}) has been raised by the following people in a somewhat less straightforward way. First, in 1979 V. Klee \cite{Kl} asked whether (\ref{G-K-W}) holds in $\Ee^2$. Then in 1987, M. Gromov \cite{Gr87} published a proof of (\ref{G-K-W}) for all $N\leq d+1$ and for not necessarily congruent balls and conjectured that his result extends to spherical $d$-space $\Sd$ (resp., hyperbolic $d$-space $\Hd$) for all $d>1$. Finally, in 1991 V. Klee and S. Wagon \cite{KlWa} asked whether (\ref{G-K-W}) holds for not necessarily congruent balls as well. (We note that in \cite{KlWa} the {\it Kneser-Poulsen conjecture} under (\ref{K-P}) is stated in its most general form, that is, for not necessarily congruent balls.) Thus, it would be proper to refer to (\ref{G-K-W}) as a special case of the {\it Gromov-Klee-Wagon problem}, which is about proving or disproving the monotonicity of the volume of intersection of not necessarily congruent balls under arbitrary contractions of their center points in $\Ee^d$, $\Ss^d$, and $\Hh^d$ for $d>1$. In any case, the author jointly with R. Connelly \cite{BeCo} confirmed (\ref{K-P}) as well as (\ref{G-K-W}) for not necessarily congruent balls when $N\leq d+3$ generalizing the relevant result of M. Gromov \cite{Gr87} for $N\leq d+1$. On the other hand, the author and R. Connelly \cite{BeCo} proved (\ref{K-P}) as well as (\ref{G-K-W}) for $N$ not necessarily congruent circular disks and for all $N>1$ in $\Ee^2$. Very recently B. Csik\'os and M. Horv\'ath \cite{CsHo} (resp., I. Gorbovickis \cite{Go}) gave a positive answer to the Gromov-Klee-Wagon problem in $\Hh^2$ (resp., $\Ss^2$ for circular disks having radii at most $\frac{\pi}{2}$). However, both  (\ref{K-P}) and (\ref{G-K-W}) remain open in $\Ed$ for all $d\geq 3$. Just very recently the author and M. Nasz\'odi \cite{BeNa}  investigated the Kneser-Poulsen conjecture and the Gromov-Klee-Wagon problem for congruent balls and for uniform contractions in $\Ed$. Following P. Pivovarov (\cite{BeNa}) we say that the (labeled) point set $Q:=\{\q_1,\dots ,\q_N\}\subset\Ee^d$ is a {\it uniform contraction} of the (labeled) point set $P:=\{\p_1,\dots ,\p_N\}\subset\Ee^d$ {\it with separating value} $\lambda>0$ in $\Ee^d$, $d>1$ if $|\q_i-\q_j|\leq\lambda\leq|\p_i-\p_j|$
holds for all $1\leq i<j\leq N$. Theorem 1.4 of \cite{BeNa} proves (\ref{G-K-W}) as well as its extension to intrinsic volumes for all uniform contractions in $\Ee^d$, $d>1$ under the condition that $ N\geq \left(1+\sqrt{2}\right)^d$. We give a short proof of this result using Theorem~\ref{Bezdek-inequality-extended} and improve it for $d\geq 42$ as follows.

\begin{theorem}\label{main-application}
Let $d>1$, $\lambda>0$, $r>0$, and $\kind$ be given and let $Q:=\{\q_1,\dots ,\q_N\}\subset\Ee^d$ be a uniform contraction of $P:=\{\p_1,\dots ,\p_N\}\subset\Ee^d$ with separating value $\lambda$ in $\Ee^d$.
\item{\bf (i)} If $1<d<42$ and $N\geq (1+\sqrt{2})^d$, then 
\begin{equation}\label{inequality-main-application}
V_{k}(P^{r})\leq V_{k}(Q^{r}).
\end{equation}
\item{\bf (ii)} If $d\geq 42$ and $N\geq \sqrt{\frac{\pi}{2d}} (1+\sqrt{2})^d+1$, then (\ref{inequality-main-application}) holds.
\end{theorem}



\section{Proof of Theorem~\ref{Bezdek-inequality-extended}}

Clearly, if $\B^r=\emptyset$, then $\A^r=\emptyset$ and (\ref{Bezdek-inequality-generalized}) follows. Similarly, it is easy to see that if $\B^r$ is a point in $\Ee^d$, then (\ref{Bezdek-inequality-generalized}) follows. Hence, we may assume that $\B^r=\B^d[\oo, R]$ and $\B=\B^d[\oo, r-R]$ with $0<R< r$.
Next recall that a special case of the Alexandrov-Fenchel inequality yields the following statement (\cite{Sc}, p. 334): if $\K$ is a convex body in $\Ee^d$ satisfying $V_i(\K)\leq V_i(\B^d[\oo, R])$ for given $1\leq i<d$ and $R>0$, then
\begin{equation}\label{A-F-inequality}
V_j(\K)\leq V_j(\B^d[\oo, R]) 
\end{equation}
holds for all $j$ with $i<j\leq d$. Thus, it is sufficient to prove (\ref{Bezdek-inequality-generalized}) for $k=1$ and $\B^r=\B^d[\oo, R]$ with $0<R< r$.

\begin{definition}
Let $\emptyset\neq K\subset\Ee^d$, $d>1$ and $r>0$. Then the {\rm $r$-ball convex hull} ${\rm conv}_rK$ of $K$ is defined by $${\rm conv}_rK:=\bigcap\{ \B^d[\x, r]\ |\ K\subseteq \B^d[\x, r]\}.$$
Moreover, let the $r$-ball convex hull of $\Ee^d$ be $\Ee^d$. Furthermore, we say that  $K\subseteq\Ee^d$ is {\rm $r$-ball convex} if $K={\rm conv}_rK$.
\end{definition}

\begin{remark}\label{empty}
We note that clearly, ${\rm conv}_rK=\emptyset$ if and only if $K^r=\emptyset$. Moreover, $\emptyset\neq K\subset\Ee^d$ is $r$-ball convex if and only if $K$ is an $r$-ball body.
\end{remark}

We need the following statement that has been proved in \cite{Be18}.

\begin{lemma}\label{basic2}
If $K\subseteq\Ee^d$, $d>1$ and $r>0$, then $K^r= ({\rm conv}_rK)^r$.
\end{lemma}
 
Hence, via an easy application of Lemma~\ref{basic2} we may assume that $\A\subset\Ee^d$ is an $r$-ball body of volume $V_{d}(\A)>0$ and $\B=\B^d[\oo, r-R]$ with $0<R< r$ such that $V_{d}(\A)=V_{d}(\B)$. Our goal is to prove that 
\begin{equation}\label{fundamental}
V_1(\A^r)\leq V_1(\B^r)=V_1(\B^d[\oo, R]).
\end{equation}

Next recall Theorem 1 of \cite{Ca}, which we state as follows.

\begin{lemma}\label{Capoyleas} 
If $\A\subset\Ee^d$ is an $r$-ball body (for $r>0$), then $\A+\A^r$ is a convex body of constant width $2r$, where $+$ denotes the Minkowski sum.
\end{lemma}

Thus, we have

\begin{corollary}\label{C-corollary}
If $\A\subset\Ee^d$ is an $r$-ball body (for $r>0$), then $$V_1(\A)+V_1(\A^r)=\frac{d\omega_d}{\omega_{d-1}}r=V_1(\B)+V_1(\B^r),$$ where
$\B=\B^d[\oo, r-R]$ and $\B^r=\B^d[\oo, R]$ with $0<R< r$ and $V_{d}(\A)=V_{d}(\B)$.
\end{corollary}

Finally, recall that (\ref{A-F-inequality}) for $j=d$ can be restated as follows (\cite{Sc}, p. 335): if $\A$ is a convex body in $\Ee^d$ satisfying $V_d(\A)= V_d(\B^d[\oo, r-R])$ for given 
$d>1$ and $0<R<r$, then 
\begin{equation}\label{reverse-inequality}
V_i(\A)\geq V_i(\B^d[\oo, r-R])
\end{equation} 
holds for all $i$ with $1\leq i<d$.

Hence, Corollary~\ref{C-corollary} and (\ref{reverse-inequality}) for $i=1$ imply (\ref{fundamental}) in a straightforward way. This completes the proof of Theorem~\ref{Bezdek-inequality-extended}.

\section{Proof of Theorem~\ref{main-application}}

\subsection{Proof of Part {\bf (i)} of Theorem~\ref{main-application}}

Let $d>1$, $\lambda>0$, $r>0$, and $\kind$ be given. If $\lambda>2r$, then $V_k(P^r)=V_k(\emptyset)=0$ and (\ref{inequality-main-application}) follows. Thus, we may assume that $0<\lambda\leq 2r$, and as in \cite{BeNa}, we proceed by proving the following theorem, which implies part {\bf (i)} of Theorem~\ref{main-application} in a straightforward way.

\begin{theorem}\label{Main-Application-Refined}
Let $d>1$, $\lambda>0$, $r>0$, and $\kind$ be given such that $0<\lambda\leq 2r$. Let $Q:=\{\q_1,\dots ,\q_N\}\subset\Ee^d$ be a uniform contraction of $P:=\{\p_1,\dots ,\p_N\}\subset\Ee^d$ with separating value $\lambda$ in $\Ee^d$. If
\item {\bf (a)} $N\geq \left(1+\frac{2r}{\lambda}\right)^d$, or
\item {\bf (b)} $0<\lambda\leq\sqrt{2} r$ and $N\geq \left(1+\sqrt{\frac{2d}{d+1}}\right)^d$,

\noindent then (\ref{inequality-main-application}) holds.
\end{theorem}

\begin{proof}

Following \cite{BeNa}, our proof is based on proper estimates of the following functionals.

\begin{definition} Let 
\begin{equation}\label{d-1}
f_{k, d}(N,\lambda ,r):=\min\{V_{k}(Q^{r})\ |\ Q:=\{\q_1,\dots ,\q_N\}\subset\Ee^d,\  |\q_i-\q_j|\leq\lambda\ {\rm for\ all} \ 1\leq i<j\leq N \}
\end{equation}
and
\begin{equation}\label{d-2}
g_{k,d}(N,\lambda , r):=\max\{V_{k}(P^{r})\ |\ P:=\{\p_1,\dots ,\p_N\}\subset\Ee^d,\  \lambda\leq |\p_i-\p_j|\ {\rm for\ all} \ 1\leq i<j\leq N \}
\end{equation}
\end{definition}

We note that in this paper the maximum of the empty set is zero. We need also

\begin{definition}
The circumradius ${\rm cr} X$ of the set $X\subseteq\Ee^d$, $d>1$  is defined by
$$
{\rm cr}X:=\inf\{r\ |\ X\subseteq \B^d[\x, r]\ {\rm for\ some}\ \x\in\Ee^d\}.
$$
\end{definition}

\bigskip

{\it Part} {\bf (a)}:
By assumption $N\geq \left(1+\frac{2r}{\lambda}\right)^d$ and so
\begin{equation}\label{ball-density}
N\left(\frac{\lambda}{2}\right)^d\kappa_d\geq\left(\frac{\lambda}{2}+r\right)^d\kappa_d,
\end{equation} 
where $\kappa_d$ denotes the volume of a $d$-dimensional unit ball in $\Ee^d$. As the balls $\p_1+\B^d[\oo, \frac{\lambda}{2}],\dots ,\p_N+\B^d[\oo, \frac{\lambda}{2}]$ are pairwise non-overlapping in $\Ee^d$ therefore 
(\ref{ball-density}) yields in a straightforward way that ${\rm cr} P> r$. Thus, $P^r=\emptyset$ and therefore clearly, $g_{k,d}(N,\lambda , r)=0\leq
f_{k,d}(N,\lambda ,r)$ holds, finishing the proof of Theorem~\ref{Main-Application-Refined}, part (a).
\bigskip

{\it Part} {\bf (b)}:
For the proof that follows we need the following straightforward extension of the rather obvious but very useful Euclidean identity (9) of \cite{BeNa}: for any $X=\{\x_1, \dots , \x_n\}\subset \Ee^d, n>1, d>1, r>0$, and $r^*>0$ one has

\begin{equation}\label{trivial}
X^r=\left(\bigcup_{i=1}^{n}\B^d[\x_i, r^*]\right)^{r+r^*}.
\end{equation}

First, we give a lower bound for (\ref{d-1}). Jung's theorem (\cite{De}) implies in a straightforward way that ${\rm cr} Q\leq\sqrt{\frac{2d}{d+1}}\frac{\lambda}{2}$ and so, $\B^d\left[\x, r-\sqrt{\frac{2d}{d+1}}\frac{\lambda}{2}\right] \subset Q^{r}$ for some $\x\in\Ee^d$. (We note that by assumption  $N\geq \left(1+\sqrt{\frac{2d}{d+1}}\right)^d>1$ and $r-\sqrt{\frac{2d}{d+1}}\frac{\lambda}{2}>r-\frac{1}{\sqrt{2}}\lambda\geq 0$.) As a result we get that 
\begin{equation}\label{1-E}
f_{k,d}(N,\lambda,r)>V_{k}\left(\B^d\left[\x, r-\sqrt{\frac{2d}{d+1}}\frac{\lambda}{2}\right]\right).
\end{equation}
Second, we give an upper bound for (\ref{d-2}). (\ref{trivial}) implies that
\begin{equation}\label{2-E}
P^{r}=\left(\bigcup_{i=1}^{N}\B^d\left[\p_i,\frac{\lambda}{2}\right]\right)^{r+\frac{\lambda}{2}}, 
\end{equation}
where the balls $\B^d[\p_1,\frac{\lambda}{2}],\dots ,\B^d[\p_N,\frac{\lambda}{2}]$ are pairwise non-overlapping in $\Ee^d$. Thus, 
\begin{equation}\label{3-E}
V_{d}\left(\bigcup_{i=1}^{N}\B^d\left[\p_i,\frac{\lambda}{2}\right]\right)=N V_{d}\left( \B^d\left[\p_1,\frac{\lambda}{2}\right]\right).
\end{equation}
Let $\mu>0$ be chosen such that $N V_{d}\left( \B^d\left[\p_1,\frac{\lambda}{2}\right]\right)=V_{d}\left( \B^d\left[\p_1,\mu\right]\right)$. Clearly,
\begin{equation}\label{4-E}
\mu=\frac{1}{2}N^{\frac{1}{d}}\lambda
\end{equation}
Now Theorem~\ref{Bezdek-inequality-extended}, (\ref{2-E}), (\ref{3-E}), and (\ref{4-E}) imply in a straightforward way that
\begin{equation}\label{5-E}
V_{k}\left(P^{r}\right)=V_{k}\left(\left(\bigcup_{i=1}^{N}\B^d\left[\p_i,\frac{\lambda}{2}\right]\right)^{r+\frac{\lambda}{2}}\right)\leq
V_{k}\left(\left(\B^d\left[\p_1, \frac{1}{2}N^{\frac{1}{d}}\lambda\right]\right)^{r+\frac{\lambda}{2}}\right)
\end{equation}
Clearly, $\left(\B^d\left[\p_1, \frac{1}{2}N^{\frac{1}{d}}\lambda\right]\right)^{r+\frac{\lambda}{2}}=\B^d\left[\p_1,r-\frac{N^{\frac{1}{d}}-1}{2}\lambda\right]$ with the convention that if $r-\frac{N^{\frac{1}{d}}-1}{2}\lambda<0$, then $\B^d\left[\p_1,r-\frac{N^{\frac{1}{d}}-1}{2}\lambda\right]=\emptyset$. Hence (\ref{5-E}) yields
\begin{equation}\label{6-E}
g_{k,d}(N,\lambda,r)\leq V_{k}\left(\B^d\left[\p_1,r-\frac{N^{\frac{1}{d}}-1}{2}\lambda\right]\right)
\end{equation}
(with $V_{k}(\emptyset)=0$). Finally, as $N\geq\left(1+\sqrt{\frac{2d}{d+1}}\right)^d$ therefore $\frac{N^{\frac{1}{d}}-1}{2}\lambda\geq \sqrt{\frac{2d}{d+1}}\frac{\lambda}{2}$ and so, (\ref{1-E}) and (\ref{6-E}) yield $g_{k,d}(N,\lambda,r)<f_{k,d}(N,\lambda,r)$, finishing the proof of Theorem~\ref{Main-Application-Refined}, part (b).
\end{proof}

\subsection{Proof of Part {\bf (ii)} of Theorem~\ref{main-application}}

The following strengthening of Theorem~\ref{Main-Application-Refined} implies part {\bf (ii)} of Theorem~\ref{main-application} in a straightforward way.
Thus, we are left to prove

\begin{theorem}\label{Main-Application-Refined and Revised}
Let $d\geq 42$, $\lambda>0$, $r>0$, and $\kind$ be given such that $0<\lambda\leq 2r$. Let $Q:=\{\q_1,\dots ,\q_N\}\subset\Ee^d$ be a uniform contraction of $P:=\{\p_1,\dots ,\p_N\}\subset\Ee^d$ with separating value $\lambda$ in $\Ee^d$. If
\item {\bf (a)} $N\geq \sqrt{\frac{\pi}{2d}}\left(1+\frac{2r}{\lambda}\right)^d+1$, or
\item {\bf (b)} $0<\lambda\leq\sqrt{2} r$ and $N\geq  \sqrt{\frac{\pi}{2d}}\left(1+\sqrt{\frac{2d}{d+1}}\right)^d+1$,

\noindent then (\ref{inequality-main-application}) holds.
\end{theorem}

\begin{proof}
We use the notations and methods of the proof of Theorem~\ref{Main-Application-Refined}. Furthermore, we need the following well-known result of U. Betke and M. Henk \cite{BeHe},
which proves the sausage conjecture of L. Fejes T\'oth in $\Ee^d$ for $d\geq 42$: whenever the balls $\B^d[\p_1,\frac{\lambda}{2}],\dots ,\B^d[\p_N,\frac{\lambda}{2}]$ are pairwise non-overlapping 
in $\Ee^d$ then
\begin{equation}\label{Betke-Henk}
V_{d}\left({\rm conv}\left(\bigcup_{i=1}^{N}\B^d\left[\p_i,\frac{\lambda}{2}\right]\right)\right)\geq (N-1)\lambda \left(\frac{\lambda}{2}\right)^{d-1}\kappa_{d-1}+\left(\frac{\lambda}{2}\right)^d\kappa_d,
\end{equation}
where ${\rm conv}(\cdot)$ denotes the convex hull of the given set. Using the inequality $\frac{\kappa_{d-1}}{\kappa_{d}}>\sqrt{\frac{d}{2\pi}}$ for $d\geq 1$ (see Lemma 1 in \cite{BeGrWi}) we get in a straightforward way from (\ref{Betke-Henk}) that
\begin{equation}\label{Betke-Henk-modified}
V_{d}\left({\rm conv}\left(\bigcup_{i=1}^{N}\B^d\left[\p_i,\frac{\lambda}{2}\right]\right)\right)>\left(\left(N-1\right)\sqrt{\frac{2d}{\pi}}+1\right)\left(\frac{\lambda}{2}\right)^d\kappa_d.
\end{equation}
\bigskip
{\it Part} {\bf (a)}:
By assumption $N\geq \sqrt{\frac{\pi}{2d}}\left(1+\frac{2r}{\lambda}\right)^d+1>\sqrt{\frac{\pi}{2d}}\left[\left(1+\frac{2r}{\lambda}\right)^d-1\right]+1$ and so
\begin{equation}\label{ball-density-1}
\left(\left(N-1\right)\sqrt{\frac{2d}{\pi}}+1\right)\left(\frac{\lambda}{2}\right)^d\kappa_d>\left(\frac{\lambda}{2}+r\right)^d\kappa_d.
\end{equation} 
As the balls $\p_1+\B^d[\oo, \frac{\lambda}{2}],\dots ,\p_N+\B^d[\oo, \frac{\lambda}{2}]$ are pairwise non-overlapping in $\Ee^d$ therefore 
(\ref{Betke-Henk-modified}) and (\ref{ball-density-1}) yield in a straightforward way that ${\rm cr} P> r$. Thus, $P^r=\emptyset$ and therefore clearly, $g_{k,d}(N,\lambda , r)=0\leq
f_{k,d}(N,\lambda ,r)$ holds, finishing the proof of Theorem~\ref{Main-Application-Refined and Revised}, part (a).

\bigskip

{\it Part} {\bf (b)}:
In the same way as in the proof of part {\bf (b)} of Theorem~\ref{Main-Application-Refined} one can derive that

\begin{equation}\label{Jung-applied}
f_{k,d}(N,\lambda,r)>V_{k}\left(\B^d\left[\x, r-\sqrt{\frac{2d}{d+1}}\frac{\lambda}{2}\right]\right).
\end{equation}

Next, recall that

\begin{equation}\label{ball-packing}
P^{r}=\left(\bigcup_{i=1}^{N}\B^d\left[\p_i,\frac{\lambda}{2}\right]\right)^{r+\frac{\lambda}{2}}, 
\end{equation}
where the balls $\B^d[\p_1,\frac{\lambda}{2}],\dots ,\B^d[\p_N,\frac{\lambda}{2}]$ are pairwise non-overlapping in $\Ee^d$. Lemma~\ref{basic2} applied to (\ref{ball-packing}) yields

\begin{equation}\label{ball-packing-upper-bounded}
P^{r}=\left(\bigcup_{i=1}^{N}\B^d\left[\p_i,\frac{\lambda}{2}\right]\right)^{r+\frac{\lambda}{2}}= \left({\rm conv}_{r+\frac{\lambda}{2}}\left(\bigcup_{i=1}^{N}\B^d\left[\p_i,\frac{\lambda}{2}\right]\right)\right)^{r+\frac{\lambda}{2}}
\subset \left({\rm conv}\left(\bigcup_{i=1}^{N}\B^d\left[\p_i,\frac{\lambda}{2}\right]\right)\right)^{r+\frac{\lambda}{2}}.
\end{equation}
Hence, Theorem~\ref{Bezdek-inequality-extended}, (\ref{Betke-Henk-modified}), and (\ref{ball-packing-upper-bounded}) imply in a straightforward way that
\begin{equation}\label{combined}
V_k(P^r)<V_k\left(\left(\B^d[\oo ,\mu]\right)^{r+\frac{\lambda}{2}}\right),
\end{equation}
where $V_d\left(\B^d[\oo ,\mu]\right)=\left(\left(N-1\right)\sqrt{\frac{2d}{\pi}}+1\right)\left(\frac{\lambda}{2}\right)^d\kappa_d$. Thus, (\ref{combined}) yields
\begin{equation}\label{final-upper-bound}
g_{k,d}(N,\lambda , r)<V_k\left(\B^d\left[\oo , r- \left(\left(\left(N-1\right)\sqrt{\frac{2d}{\pi}}+1\right)^{\frac{1}{d}}-1\right)\frac{\lambda}{2}\right]\right)
\end{equation}
(with $V_{k}(\emptyset)=0$). Finally, as $N\geq  \sqrt{\frac{\pi}{2d}}\left(1+\sqrt{\frac{2d}{d+1}}\right)^d+1>\sqrt{\frac{\pi}{2d}}\left(1+\sqrt{\frac{2d}{d+1}}\right)^d+\left(1-\sqrt{\frac{\pi}{2d}}\right)$ therefore $\left(\left(\left(N-1\right)\sqrt{\frac{2d}{\pi}}+1\right)^{\frac{1}{d}}-1\right)\frac{\lambda}{2}>\sqrt{\frac{2d}{d+1}}\frac{\lambda}{2}$ and so, (\ref{Jung-applied}) and (\ref{final-upper-bound}) yield $g_{k,d}(N,\lambda,r)<f_{k,d}(N,\lambda,r)$, finishing the proof of Theorem~\ref{Main-Application-Refined and Revised}, part (b).

\end{proof}

{\bf Acknowledgements}

The author is indebted to the anonymous referees for careful reading and valuable comments.

\bigskip


\noindent K\'aroly Bezdek \\
\small{Department of Mathematics and Statistics, University of Calgary, Canada}\\
\small{Department of Mathematics, University of Pannonia, Veszpr\'em, Hungary\\
\small{E-mail: \texttt{bezdek@math.ucalgary.ca}}

\end{document}